\documentclass[11pt]{amsart}
\usepackage{amsfonts}
\usepackage{ifthen}
\usepackage{amsthm}
\usepackage{amsmath}
\usepackage{graphicx}
\usepackage{amscd,amssymb,amsthm}
\usepackage{graphicx}
\usepackage{color}

\newcounter{minutes}
\divide\time by 60
\newcounter{hours}
\multiply\time by 60 \addtocounter{minutes}{-\time}

\usepackage{psfrag}
\usepackage{amssymb}
\usepackage{hyperref}
\usepackage{graphicx}

\date{}
\newfont{\cyrilic}{wncyr10 scaled 1000}

\title[Convexity properties of generalized trigonometric functions]{Convexity properties of generalized trigonometric and hyperbolic functions}


\author[\'A. Baricz]{\'Arp\'ad Baricz}
\address{Department of Economics, Babe\c{s}-Bolyai University, 400591 Cluj-Napoca, Romania}
\email{bariczocsi@yahoo.com}

\author[B.A. Bhayo]{Barkat Ali Bhayo}
\address{Department of Mathematical Information Technology, University of Jyv\"askyl\"a, 40014 Jyv\"askyl\"a, Finland}
\email{bhayo.barkat@gmail.com}

\author[R. Kl\'en]{Riku Kl\'en}
\address{Department of Mathematics and Statistics, University of Turku, 20014 Turku, Finland}
\email{riku.klen@utu.fi}

\newtheorem{theorem}{Theorem}
\newtheorem{lemma}{Lemma}

\begin{document}

\def\thefootnote{}
\footnotetext{ \texttt{File:~\jobname .tex,
          printed: \number\year-0\number\month-\number\day,
          \thehours.\ifnum\theminutes<10{0}\fi\theminutes}
} \makeatletter\def\thefootnote{\@arabic\c@footnote}\makeatother


\begin{abstract}
We study the power mean inequality for the generalized trigonometric and hyperbolic functions
with two parameters. The generalized $p$-trigonometric and $(p,q)$-trigonometric functions
were introduced by Lindqvist and Takeuchi, respectively.
\end{abstract}


\maketitle



\section{\bf Introduction and Main Results}
The generalized trigonometric functions were introduced by Lindqvist \cite{lp} two decades ago. These $p$-trigonometric functions, $p > 1$, coincide with the usual trigonometric functions for $p=2$. Recently, the $p$-trigonometric functions have been studied extensively, see for example \cite{bv,bem,dm,le} and their references. Dr\'abek and Man\'asevich \cite{dm} considered a certain ($p,q$)-eigenvalue problem with the Dirichl\'et boundary condition and found the complete solution to the problem. The solution of a special case is the function $\sin_{p,q}$, which is the first example of so called $(p,q)$-trigonometric function. Motivated by the ($p,q$)-eigenvalue problem Takeuchi \cite{t} introduced the $(p,q)$-trigonometric functions, $p,q > 1$. These $(p,q)$-trigonometric functions have recently been studied also in \cite{bv2,bv3,be,cjw,egl}, and the functions agree with the $p$-trigonometric functions for $p=q$.

The \emph{Gaussian hypergeometric function} is the
analytic continuation to the slit plane $\mathbb{C}\setminus[1,\infty)$ of the series
$$F(a,b;c;z)={}_2F_1(a,b;c;z)=\sum_{n\geq0}\frac{(a,n)(b,n)}
{(c,n)}\frac{z^n}{n!},\qquad |z|<1,$$
for given complex numbers $a,b$ and $c$ with $c\neq0,-1,-2,\ldots$ .
Here $(a,0)=1$ for $a\neq 0$, and $(a,n)$ is the \emph{shifted factorial function} or the \emph{Appell symbol}
$$(a,n)=a(a+1)(a+2)\cdots(a+n-1)$$
for $n=1,2,\dots \,.$
The hypergeometric function is used to define the $(p,q)$-trigonometric functions and it has numerous special functions as its
special or limiting cases \cite{AS}.

For $p,q > 1$ and $x\in (0,1)$ the function ${\rm arcsin}_{p,q}$ is defined by
\[
  {\arcsin}_{p,q}(x)=\int^x_0(1-t^q)^{-1/p}dt= x\,F\left(\frac{1}{p},\frac{1}{q};1+\frac{1}{q};x^q\right).
\]
We also define for $x\in(0,1)$
${\arccos}_{p,q}(x)={\rm arcsin}_{p,q}((1-x^p)^{1/q})$
(see \cite[Prop. 3.1]{egl}) and for $x\in (0,\infty)$
\begin{eqnarray*}{\rm arsinh}_{p,q}(x)&=&\int^x_0(1+t^q)^{-1/p}dt=
x\,F\left(\frac{1}{p},\frac{1}{q};1+\frac{1}{q};-x^q\right)\\
&=&\left(\frac{x^p}{1+x^q}\right)^{1/p}F\left(1,\frac{1}{p},1+\frac{1}{q};\frac{x^q}{1+x^q}\right).
\end{eqnarray*}

Their inverse functions are given by
$$\sin_{p,q}, \cos_{p,q}:(0,\pi_{p,q}/2)\to (0,1),\ \ \sinh_{p,q}:(0,\infty)\to (0,\infty),$$
where $\pi_{p,q}/2 = \arcsin_{p,q}(1)$. Now, we define $\tan_{p,q}:(0,\pi_{p,q}/2)\to (0,\infty)$ by
$$\tan_{p,q}(x)=\frac{\sin_{p,q}(x)}{\cos_{p,q}(x)},\quad \cos_{p,q}(x)\neq 0,$$
and we denote its inverse by ${\arctan}_{p,q}$.

We introduce next a geometric definition for the functions $\sin_{p,q}$ and $\cos_{p,q}$. The definition is based on the formula \cite[eq. (2.7)]{egl}
\begin{equation}\label{thm:sinpq+cospc}
  |\sin_{p,q}(x)|^q + |\cos_{p,q}(x)|^p = 1.
\end{equation}
The usual trigonometric functions are geometrically defined by the unit circle. In the same manner we can define the $(p,q)$-trigonometric functions. Instead of the unit circle we use the Lam\'e-curve (or generalized superellipse) defined by
\[
  C = \{ (x,y) \in \mathbb{R}^2 \colon x=\cos^{2/p} t, \, y=\sin^{2/q} t, \, t \in  [0,\pi/2] \}.
\]
Because $(\sin^{2/q} t)^q + (\cos^{2/p} t)^p = 1$ it is clear by \eqref{thm:sinpq+cospc} that $C=D$ for
\[
  D = \{ (x,y) \in \mathbb{R}^2 \colon x=\sin_{p,q}(t), \, y=\cos_{p,q}(t),\, t\in [0,\pi_{p,q}/2] \}.
\]
The curve $C$ is defined only in the first quadrant, but by reflections we can easily extend $C$ to form a circular curve, see Figure \ref{fig1}.

\begin{center}
\begin{figure}[ht]
\psfrag{A}{$\sin_{p,q} \, x$}
\psfrag{B}{$\cos_{p,q}\, x$}
\psfrag{C}{$C$}
\includegraphics[width=0.3 \textwidth]{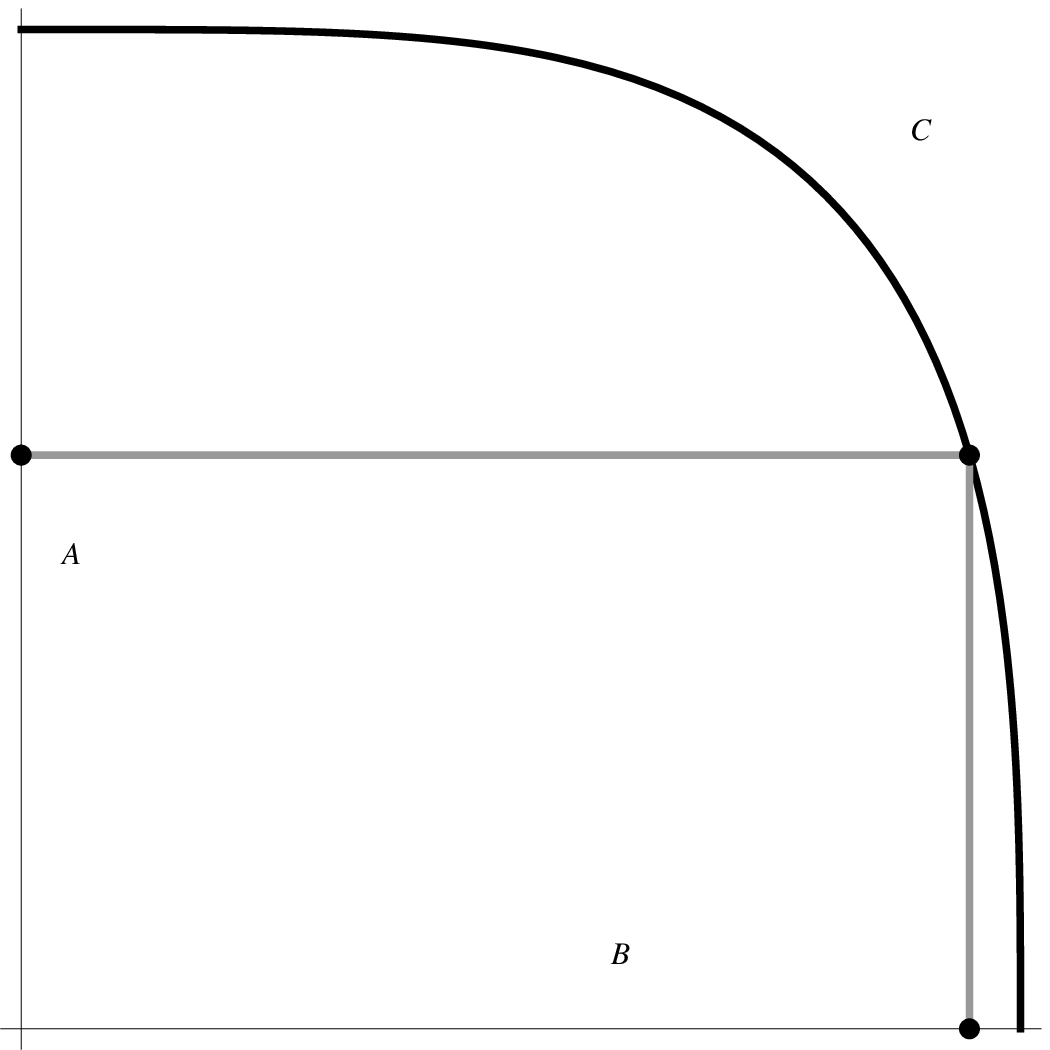}\hspace{2cm}
\includegraphics[width=0.3 \textwidth]{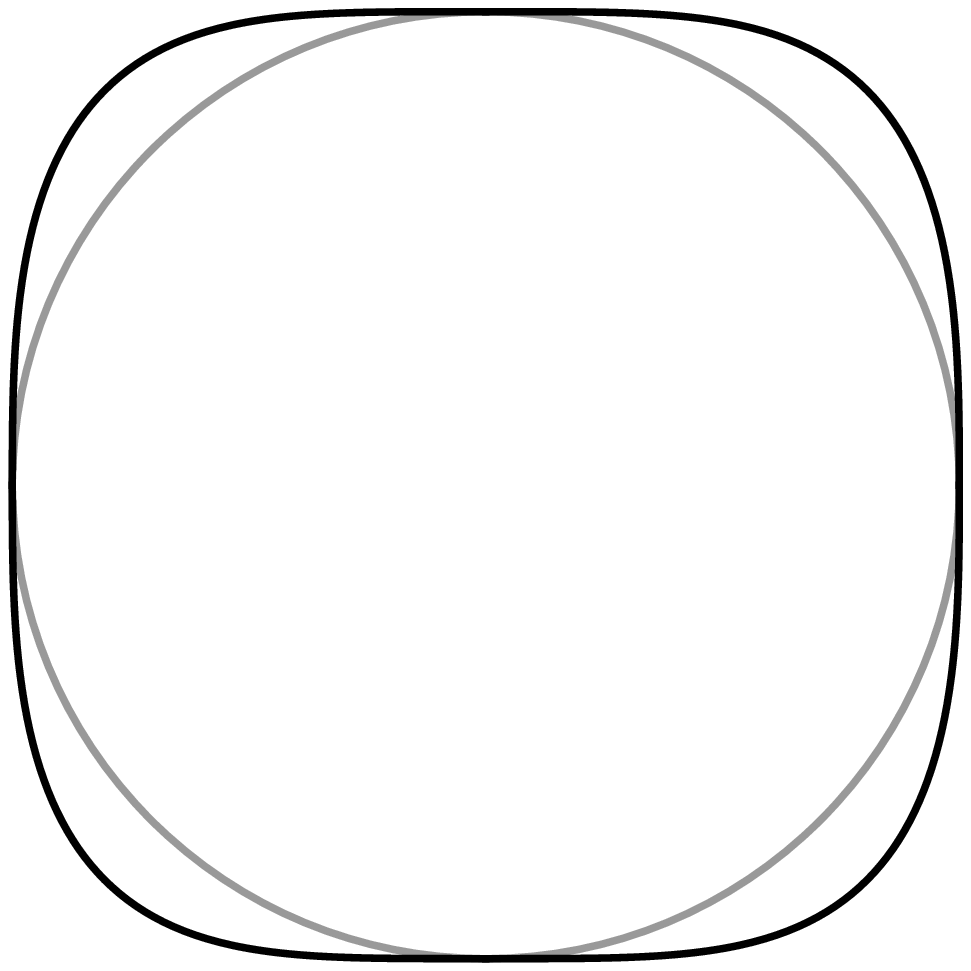}
\hspace{-1cm}\caption{Left: Curve $C$ for $(p,q) =(4,3)$ and points $(0,\sin_{4,3} \, x)$, $(\cos_{4,3} \, x,0)$ and $(\cos_{4,3} \, x,\sin_{4,3} \, x)$. Right: Curve $C$ for $(p,q)=(4,3)$ extended to all four quadrants (black) and the unit circle (gray). \label{fig1}}
\end{figure}
\end{center}

Now, we recall the definition of convex functions with respect to power means. For $a\in\mathbb{R}$ and $x,y>0$, the power mean $M_a$ of order $a$ is defined by
$$M_a(x,y)=\displaystyle\left\{\begin{array}{lll} \displaystyle\left(\frac{x^a+y^a}{2}\right)^{1/a},\;\quad a\neq 0,\\
\sqrt{x\,y}\,, \quad\qquad\qquad a=0\,.\end{array}\right.$$
We consider the function $f:I\subset(0,\infty)\to (0,\infty),$ which is continuous, and let $M_a(x,y)$
and $M_b(x,y)$ be the power means of order $a$ and $b$ of $x>0$ and $y>0$. For $a,b\in\mathbb{R}$ we say that $f$ is $M_aM_b$-convex (concave) or simply $(a,b)$-convex (concave), if for those $a$ and $b$ we have
$$f (M_a(x, y)) \leq (\geq) M_b(f(x),f(y)) \,\, \text{ for all} \,\, x,y \in I\,.$$
It is important to mention here that $(1,1)$-convexity means the usual convexity, $(1,0)$ means the log-convexity and $(0,0)$-convexity is the so-called geometrical (multiplicative) convexity. Moreover, it is known that if the function $f$ is differentiable, then it is $(a,b)$-convex (concave) if and only if $x\mapsto x^{1-a}f'(x)[f(x)]^{b-1}$ is increasing (decreasing). See \cite{bari} for more details.

In the next results we can see that the above mentioned $(p,q)$ generalized trigonometric functions are $(a,a)$-convex/concave. These results refine the earlier results in \cite{bv2}.

\begin{theorem}\label{thm1} If $p,q>1$ and $a\geq 1,$ then $\arcsin_{p,q}$ is $(a,a)$-convex on $(0,1),$ $\arctan_{p,q}$ is $(a,a)$-concave on $(0,1),$ while ${\rm arcsinh}_{p,q}$ is $(a,a)$-concave on $(0,\infty).$ In other words, if $p,q>1$ and $a\geq 1$, then we have
\begin{eqnarray*}
{\arcsin}_{p,q}(M_a(r,s)) & \leq & M_a({\rm arcsin}_{p,q}(r),{\rm arcsin}_{p,q}(s))\,,\quad r,s\in(0,1),\\
{\arctan}_{p,q}(M_a(r,s)) & \geq & M_a({\rm arctan}_p(r),{\rm arctan}_p(s))\,,\quad r,s\in(0,1),\\
{\rm arsinh}_{p,q}(M_a(r,s)) & \geq & M_a({\rm arsinh}_{p,q}(r),{\rm arsinh}_{p,q}(s))\,,\quad r,s>0.
\end{eqnarray*}
\end{theorem}

\begin{theorem}\label{thm2} If $p,q>1$ and $a\geq 1,$ then $\sin_{p,q}$ is $(a,a)$-concave,
and $\cos_{p,q},$ $\tan_{p,q},$ $\sinh_{p,q}$ are $(a,a)$-convex on $(0,1).$ In other words, if $p,q>1,$ $a\geq 1$ and $r,s\in(0,1)$, then the next inequalities are valid
\begin{eqnarray*}
  {\sin}_{p,q}(M_a(r,s)) & \geq & M_a({\rm sin}_{p,q}(r),{\rm sin}_{p,q}(s)),\\
  {\cos}_{p,q}(M_a(r,s)) & \leq & M_a({\rm cos}_{p,q}(r),{\rm cos}_{p,q}(s)),\\
  {\tan}_{p,q}(M_a(r,s)) & \leq & M_a({\rm tan}_{p,q}(r),{\rm tan}_{p,q}(s)),\\
  {\sinh}_{p,q}(M_a(r,s)) & \leq & M_a({\rm sinh}_{p,q}(r),{\rm arsinh}_{p,q}(s)).
\end{eqnarray*}
\end{theorem}

The next theorems improve some of the above results.

\begin{theorem}\label{nthm1} If $p,q>1$, $a\leq 0$ and $b\in \mathbb{R}$ or $0< a\leq b$ and $b\leq 1$, then ${\arcsin}_{p,q}$ is $(a,b)$-convex on $(0,1),$ and in particular if $p=q$, then the function $\arcsin_p=\arcsin_{p,p}$ is $(a,b)$-convex on $(0,1).$ In other words, if $p,q>1$, $a\leq 0$ and $b\in \mathbb{R}$ or $0< a\leq b$ and $b\leq 1$, then for all $r,s\in(0,1)$ we have
$${\arcsin}_{p,q}(M_a(r,s))\leq M_b({\rm arcsin}_{p,q}(r),{\rm arcsin}_{p,q}(s)).$$
\end{theorem}

\begin{theorem}\label{nthm2} If $p,q>1$, $a\leq 0\geq b$ or $0< b\leq a$ and $a\leq 1$, then ${\rm arsinh}_{p,q}$ is
$(a,b)$-concave on $(0,\infty),$ and in particular if $p=q$, then the function ${\rm arsinh}_{p}={\rm arsinh}_{p,p}$ is $(a,b)$-concave on $(0,\infty).$ In other words, if $p,q>1$, $a\leq 0\geq b$ or $0< b\leq a$ and $a\leq 1$, then for all $r,s\in(0,\infty)$ we have
$${\rm arsinh}_{p,q}(M_a(r,s))\geq M_b({\rm arsinh}_{p,q}(r),{\rm arsinh}_{p,q}(s)).$$
\end{theorem}

It is worth to mention that Chu et al. \cite{cjw} very recently proved implicitly the $(p_1,0)$-convexity of $\arcsin_{p,q}$ for $p_1\leq 0,$ $p,q>1,$ and the $(p_2,0)$-concavity of ${\rm arcsinh}_{p,q}$ for $p_2\geq0,$ $p,q>1.$ Our approach is similar but somewhat easier, since we applied in the above cases directly the results from \cite{bari}. We also mention that recently Jiang and Qi \cite{jq} proved that the $\arcsin_{p,q}$ is $(0,0)$-convex, by using the same idea as we used in this paper. See also \cite{barbarkvuo,bv} for more details.

Now, observe that by using the change of variable $u=(1-t^q)^{1/p}$
$${\rm arccos}_{p,q}(x)={\rm arcsin}_{p,q}(y)=\int_0^y(1-t^p)^{-1/p}dt,\quad {\rm where}\quad y=(1-x^p)^{1/q},$$
can be written as
$${\rm arccos}_{p,q}(x)=\frac{p}{q}\int_x^1f(u)du, \quad \mbox{where}\quad f(u)=(1-u^p)^{1/q-1}u^{p-2}.$$
Consequently, we have
$$\int_0^xf(u)du=\frac{q}{p}{\arccos}_{p,q}(0)-\frac{q}{p}{\arccos}_{p,q}(x)=\frac{q}{p}{\arcsin}_{p,q}(1)-\frac{q}{p}{\arccos}_{p,q}(x).$$

The next result is about this integral.

\begin{theorem}\label{nthm3}
The function $x\mapsto \pi_{p,q}/2-{\rm arccos}_{p,q}(x)$ is $(a,b)$-convex on $(0,1)$ if $p\in(1,2],$ $q>1,$ $a<0$ and
$b\in\mathbb{R}$ or if $p,q>1,$ $a\leq 0$ and $b\geq 0.$ In other words, if $p\in(1,2],$ $q>1,$ $a<0$ and
$b\in\mathbb{R}$ or if $p,q>1,$ $a\leq 0$ and $b\geq 0,$ then for all $r,s\in(0,1)$ we have
$$\pi_{p,q}/2-\arccos_{p,q}(M_a(r,s))\leq M_b(\pi_{p,q}/2-\arccos_{p,q}(r),\pi_{p,q}/2-\arccos_{p,q}(s)).$$
\end{theorem}

It is worth to mention that if in the above theorem we take $a=b=0,$ then we have that the function $x\mapsto {\rm arcsin}_{p,q}(1)-{\rm arccos}_{p,q}(x)$ is geometrically convex on $(0,1)$ for $p,q>1$, and in particular if $p=q,$ then the function $x\mapsto {\rm arcsin}_{p,p}(1)-{\rm arccos}_{p,p}(x)$ is geometrically convex on $(0,1)$ for $p>1$. In other words, for $p,q>1$ and $r,s\in(0,1)$ we have
$${\pi_{p,q}}/{2}-{\arccos}_{p,q}(\sqrt{rs})\leq \sqrt{(\pi_{p,q}/{2}-{\arccos}_{p,q}(r))(\pi_{p,q}/{2}-{\arccos}_{p,q}(s))}.$$

\section{\bf Preliminary results}

In this section we present some preliminary results which will be used in the proof of the main theorems. The first known result considers the $(p,q)$-trigonometric functions whereas the other results consider properties of real functions.

\begin{lemma}\label{jan}\cite{egl}
For all $p,q>1$ and $x\in(0,\pi_{p,q}/2)$, we have
$$\left( \sin_{p,q}(x)\right)'=\cos_{p,q}(x),$$
$$\left(\cos_{p,q}(x)\right)'=-\frac{p}{q}(\cos_{p,q}(x))^{2-p}(\sin_{p,q}(x))^{q-1},$$
$$\left(\tan_{p,q}(x)\right)'=1+\frac{p\,(\sin_{p,q}(x))^{q}}{q\,(\cos_{p,q}(x))^{p}}.$$
\end{lemma}

The following result is known as \emph{the monotone form of l'Hospital's rule}.

\begin{lemma}\label{125}\cite[Theorem 1.25]{avvb}
For $-\infty<a<b<\infty$,
let $f,g:[a,b]\to \mathbb{R}$
be continuous on $[a,b]$, and be differentiable on
$(a,b)$. Let $g'(x)\neq 0$
on $(a,b)$. If $f'/g'$ is increasing
(decreasing) on $(a,b)$, then so are
$$x\mapsto\frac{f(x)-f(a)}{g(x)-g(a)}\quad and \quad x\mapsto \frac{f(x)-f(b)}{g(x)-g(b)}.$$
If $f'/g'$ is strictly monotone,
then the monotonicity in the conclusion
is also strict.
\end{lemma}

The following two results consider the convexity of the inverse function.

\begin{lemma}\label{lemku}\cite[Theorem 2]{ku}
Let $J\subset\mathbb{R}$ be an open interval, and let $f:J\to \mathbb{R}$
be strictly monotonic function. Let $f^{-1}:f(J)\to J$ be the inverse to $f$.
If $f$ is convex and increasing, then $f^{-1}$ is concave.
\end{lemma}

\begin{lemma}\label{lemmr}\cite[Proposition 2]{m}
Let $f \colon (a,b) \to (c,d) = f(a,b) \subset \mathbb{R}$ be a convex function and let $f^{-1} \colon (c,d) \to \mathbb{R}$ be its inverse. If $f$ is increasing (decreasing), then $f^{-1}$ is increasing (decreasing) and concave (convex).
\end{lemma}

The following two results consider the $(p,q)$-convexity of the function.

\begin{lemma}\cite[Lemma 3]{bari}\label{bari lemma 3}
  Let $a,b \in \mathbb{R}$ and let $f \colon [\alpha,\beta] \to (0,\infty)$ be a differentiable function for $\alpha,\beta \in (0,\infty)$. The function $f$ is $(a,b)$-convex ($(a,b)$-concave) if and only if $x \mapsto x^{1-a} f'(x) (f(x))^{b-1}$ is increasing (decreasing).
\end{lemma}

\begin{lemma}\cite[Theorem 7]{bari}\label{bari theorem 7}
  Let $\alpha,\beta \in (0,\infty)$ and $f \colon [\alpha,\beta] \to (0,\infty)$ be a differentiable function. Denote $g(x) = \int_\alpha^x f(t) \, dt$ and $h(x) = \int_x^\beta f(t) \, dt$. Then
  
  \noindent (a) If for all $a \in [0,1]$ the function $f$ is $(a,0)$-concave, then the function $g$ is $(a,b)$-concave for all $a \in [0,1]$ and $b \le 0$. If, in addition the function $x \mapsto x^{1-a} f(x)$ is increasing for all $a \in [0,1]$, then $g$ is $(a,b)$-concave for all $a \in [0,1]$ and $b \in (0,1)$. Moreover, if for all $a \in \mathbb{R}$ the function $x \mapsto x^{1-a} f(x)$ is increasing, then $g$ is $(a,b)$-convex for all $a \in \mathbb{R}$ and $b \ge 1$.
  
  \noindent (b) If for all $a \in [0,1]$ the function $f$ is $(a,0)$-concave, then the function $g$ is $(a,b)$-concave for all $a \in [0,1]$ and $b \le 0$. If, in addition the function $x \mapsto x^{1-a} f(x)$ is decreasing for all $a \in [0,1]$, then $g$ is $(a,b)$-concave for all $a \in [0,1]$ and $b \in (0,1)$. Moreover, if for all $a \in \mathbb{R}$ the function $x \mapsto x^{1-a} f(x)$ is decreasing, then $g$ is $(a,b)$-convex for all $a \in \mathbb{R}$ and $b \ge 1$.
  
  \noindent (c) If for all $a \notin (0,1)$ we have $\alpha^{1-a} f(\alpha) = 0$ and the function $f$ is $(a,0)$-convex, then $g$ is $(a,b)$-convex for all $a \notin (0,1)$ and $b \ge 0$. If, in addition the function $x \mapsto x^{1-a} f(x)$ is increasing for all $a \notin (0,1)$, then $g$ is $(a,b)$-convex for all $a \notin (0,1)$ and $b < 0$.
  
  \noindent (d) If for all $a \notin (0,1)$ we have $\beta^{1-a} f(\beta) = 0$ and the function $f$ is $(a,0)$-convex, then $g$ is $(a,b)$-convex for all $a \notin (0,1)$ and $b \ge 0$. If, in addition the function $x \mapsto x^{1-a} f(x)$ is decreasing for all $a \notin (0,1)$, then $g$ is $(a,b)$-convex for all $a \notin (0,1)$ and $b < 0$.
\end{lemma}

Next we introduce a convexity result for the functions $\tan_{p,q}$ and ${\arctan}_{p,q}$.

\begin{lemma}\label{art}
If $p,q>1$, then the function $\tan_{p,q}$
is increasing and convex on $(0,\pi_{p,q}/2),$ while the function ${\arctan}_{p,q}$
is increasing and concave on $(0,1).$
\end{lemma}

\begin{proof}[\bf Proof] By Lemma \ref{jan} we get
\[
  \left[\tan_{p,q}(x)\right]'' = 1+\frac{p}{q^2}\left[\frac{pq(\sin_{p,q}(x))^{q-1}(\cos_{p,q}(x))^{1+p}+p^2\cos_{p,q}(x)(\sin_{p,q}(x))^{2q-1}}
{(\cos_{p,q}(x))^{2p}}\right],
\]
which is positive since both $\sin_{p,q}$ and $\cos_{p,q}$ are positive. Hence $\tan_{p,q}$ is convex. By Lemma \ref{jan} we observe that the derivative of $\tan_{p,q}$ is positive and thus $\tan_{p,q}$ is increasing. It follows from Lemma \ref{lemmr} that ${\arctan_{p,q}}$ is increasing and concave.
\end{proof}

Now, we focus on some monotonicity results.

\begin{lemma}\label{lem1} Let $a\geq0,$ $p,q>1$ and consider the functions $f,h:(0,1)\to\mathbb{R}$ and $g:(0,\infty)\to \mathbb{R},$ defined by
$$f(x)=\left(\frac{{\arcsin}_{p,q}(x)}{x}\right)^{a}({\arcsin}_{p,q}(x))',$$
$$g(x)=\left(\frac{{\rm arsinh}_{p,q}(x)}{x}\right)^{a}({\rm arsinh}_{p,q}(x))',$$
$$h(x)=\left(\frac{{\arctan}_{p,q}(x)}{x}\right)^{a}({\arctan}_{p,q}(x))'.$$
Then the function $f$ is increasing and the functions $g$ and $h$ are decreasing.
\end{lemma}

\begin{proof}[\bf Proof] It follows from the definition that
\[x\mapsto ({\rm arcsin}_{p,q}(x))'=(1-x^q)^{-1/p},\]
is increasing since
\[
  ({\rm arcsin}_{p,q}(x))''=\frac{q x^{q-1}(1-x^q)^{-1-1/p}}{p} > 0.
\]
By Lemma \ref{125}, the function $x\mapsto({\arcsin}_{p,q}\,x)/x$ is increasing $(0,1)$ and consequently the function $x\mapsto (({\arcsin}_{p,q}\,x)/x)^a$ is increasing too on $(0,1).$ Since the product of increasing functions is increasing, the function $f$ is increasing. Now, the function ${\arctan}_{p,q}$ is increasing and concave $(0,1)$ by Lemma \ref{art}, which implies that $x\mapsto({\rm arctan}_{p,q}(x))'$ is decreasing on $(0,1).$ Using again Lemma \ref{125} it follows that $x\mapsto({\arctan}_{p,q}(x))/x$ is decreasing on $(0,1)$, and hence is the function $h$. Finally, the proof for function $g$ follows similarly as the proof for $f$, because
\[
  ({\rm arsinh}_{p,q}(x))''=\frac{-q x^{q-1}(1+x^q)^{-1-1/p}}{p} < 0.\qedhere
\]
\end{proof}

\begin{lemma}\label{kuclem} If $a\geq 0$ and $p,q>1,$ then the function
$$x\mapsto\left(\frac{{\sin}_{p,q}(x)}{x}\right)^{a}({\sin}_{p,q}(x))'$$
is decreasing on $(0,1),$ while the functions
$$x\mapsto\left(\frac{{\cos}_{p,q}(x)}{x}\right)^{a}({\cos}_{p,q}(x))',\ x\mapsto\left(\frac{{\tan}_{p,q}(x)}{x}\right)^{a}({\tan}_{p,q}(x))'$$
and $$x\mapsto\left(\frac{{\sinh}_{p,q}(x)}{x}\right)^{a}({\sinh}_{p,q}(x))'$$
are increasing on $(0,1).$
\end{lemma}

\begin{proof}[\bf Proof] Let $f(x)={\arcsin}_{p,q}(x),$ where $x\in(0,1)$. We get
$$f'(x)=\frac{1}{(1-x^q)^{1/p}},$$
which is positive and increasing, hence $f$ is convex. Clearly
$\sin_{p,q}$ is increasing, and by Lemma \ref{lemku} is concave. This
implies that $x\mapsto (\sin_{p,q}(x))'$ is decreasing, and $x\mapsto(\sin_{p,q}
(x))/x$ is also decreasing by Lemma \ref{125}.

Similarly we get that
$x\mapsto(\cos_{p,q}(x))'$, $x\mapsto (\tan_{p,q}(x))'$ and $x\mapsto (\sinh_{p,q}(x))'$ are increasing,
and the assertion follows from Lemma \ref{125}.
\end{proof}

\section{\bf Proofs of the main results}

\begin{proof}[\bf Proof of the Theorem \ref{thm1}.] \rm Let $0<x<y<1$, and $u=((x^a+y^a)/2)^{1/a}>x$. Let us denote $h_1(x) = {\arcsin}_{p,q}(x)$,
$h_2(x) = {\arctan}_{p,q}(x)$, $h_3(x) = {\rm arsinh}_{p,q}(x)$ and for $i\in\{1,2,3\}$ define
\[
  g_i(x)=h_i(u)^a-\frac{h_i(x)^a+h_i(y)^a}{2}.
\]
Differentiating with respect to $x$, we get $du/dx=(1/2)(x/u)^{a-1}$ and
\begin{eqnarray*}
g_i'(x)&=&\frac{1}{2}\,a\,h_i(x)^{a-1}\frac{d}{dx}(h_i(u))\left(\frac{x}{u}\right)^{a-1}
-\frac{1}{2}\,a\,h_i(x)^{a-1}\frac{d}{dx}(h_i(x))\\
&=&\frac{a}{2}x^{a-1}(f_i(u)-f_i(x)),
\end{eqnarray*}
where
\[
  f_i(x)=\left(\frac{h_i(x)}{x}\right)^{a-1}\frac{d}{dx}(h_i(x)).
\]
By Lemma \ref{lem1} $g_1'$ is positive and $g_2',$ $g_3'$ are negative. Hence $g_1$ is increasing and $g_2,$ $g_3$ are decreasing. This implies that
\[
  g_1(x) < g_1(y)=0, \quad g_2(x) > g_2(y)=0, \quad g_3(x) > g_3(y)=0,
\]
and the assertion follows.
\end{proof}

\begin{proof}[\bf Proof of the Theorem \ref{thm2}.]
The proof is similar to the proof of Theorem \ref{thm1} and follows from Lemma \ref{kuclem}.
\end{proof}

\begin{proof}[\bf Proof of Theorem \ref{nthm1}] Let us consider the function $f:(0,1)\to (0,\infty),\,f(t)=(1-t^q)^{-1/p}$.
If $a\leq 0$, then the function
$$t\mapsto\frac{t^{1-a}f'(t)}{f(t)}=\frac{q}{p}\frac{t^{q-a}}{1-t^q}$$
is increasing on $(0,1)$, that is $f$ is $(a,0)$-convex on $(0,1)$, according to Lemma \ref{bari lemma 3}.
Since $t^{1-a}(1-t^q)^{-1/p}\to 0$, as $t\to 0$, and $t\to t^{1-a}(1-t^q)^{-1/p}$ is increasing for $a\leq 0$ as a product of two increasing and positive functions from Lemma \ref{bari theorem 7}, we deduce that $x\mapsto {\rm arcsin}_{p,q}(x)$ is $(a,b)$-convex for $a\leq 0$ and $b\geq 0$ or $a\leq 0$ and $b<0$, that is, for $a\leq 0$ and $b\in \mathbb{R}$.

According to Theorem \ref{thm1}, the function $x\mapsto {\rm arcsin}_{p,q}(x)$ is $(a,a)$-convex for $a\geq 1$ and $p,q>1$. Taking into account that the power mean is increasing with respect to its order it follows that $x\mapsto {\rm arcsin}_{p,q}(x)$ is $(a,b)$-convex on $(0,1)$ for $1\leq a\leq b$ and $p,q>1$.

Let $u,v:(0,1)\to (0,\infty)$ be the functions defined by
$u(x)=x^{1-a}$ and $v(x)=\varphi(x)(\varphi(x))^{b-1}$ where $\varphi(x)={\rm arcsin}_{p,q}(x).$ Then we have that $u'(x)=(1-a)x^{-a}\geq 0$ if
$a\leq 1$, $x\in(0,1)$, and
$$v'(x)=(\varphi(x))^{b-2}(\varphi''(x)\varphi(x)+(b-1)(\varphi'(x))^2)\geq 0$$
if $b\geq 1$ and $x\in(0,1)$, since $\varphi$ is convex. Thus we obtain $x\mapsto x^{1-a}\varphi'(x)(\varphi(x))^{b-1}$ is decreasing on $(0,1)$
if $a\leq 1$ and $b\geq 1$. According to Lemma \ref{bari lemma 3} it follows that $x\mapsto {\rm arcsin}_{p,q}(x)$ is $(a,b)$-convex on $(0,1)$ for $a\leq 1$ and $b\geq 1.$

Finally, we would like to show that the $(a,b)$-convexity of $x\mapsto {\rm arcsin}_{p,q}(x)$ is the case when $b\geq a$ and $b\geq 1$ can be
proved by a somewhat different approach. For this, consider the function $\phi:(0,1)\to (0,\infty)$, defined by
$\phi(x)=x^{1-a}\varphi'(x)(\varphi(x))^{b-1}$. Then we have
$$\Delta(x)=\frac{x\phi'(x)}{\phi(x)}=1-a+\frac{q}{p}\frac{x^q}{1-x^q}+(b-1)\frac{x\varphi'(x)}{\varphi(x)},$$
and
$$\Delta'(x)=\frac{q}{p}\frac{qx^{q-1}}{(1-x^q)^2}+(b-1)\left(\frac{x\varphi'(x)}{\varphi(x)}\right)'\geq 0$$
if $x\in(0,1)$ and $b\geq 1$, since $\varphi$ is $(0,0)$-convex, as we proved before. Consequently we have $\Delta(x)\geq \Delta(0)=b-a\geq 0$
if $x\in(0,1)$ and $b\geq a$. In other words, if $b\geq a$ and $b\geq 1$, then $x\mapsto {\rm arcsin}_{p,q}(x)$ is $(a,b)$-convex on $(0,1)$.
\end{proof}

\begin{proof}[\bf Proof of Theorem \ref{nthm2}] Observe that for $f:(0,\infty)\to(0,\infty),\,f(t)=(1+t^q)^{-1/p}$, we have
$$\left(\frac{tf'(t)}{f(t)}\right)'=-\frac{q^2}{p}\frac{t^{q-1}}{(1+t^q)^2}<0\quad \mbox{for all}\quad t>0\quad \mbox{and} \quad p,q>1.$$
Consequently, $f$ is geometrically concave on $(0,\infty)$ and according to Lemma \ref{bari theorem 7} the function
$$x\mapsto {\rm arsinh}_{p,q}(x)=\int_0^x(1+t^q)^{-1/p}dt$$ is also geometrically concave (or (0,0)-concave) on $(0,\infty)$.

Now, consider the functions $u,v,w:(0,\infty)\to(0,\infty)$, defined by
$$u(x)=\frac{x\varphi'(x)}{\varphi(x)},\quad v(x)=x^{-a},\quad w(x)=(\varphi(x))^b,$$
where $\varphi(x)={\rm arsinh}_{p,q}(x)$. Observe that for $a\geq 0$ and $b\leq 0$, the function $v$ and $w$ are decreasing, and hence in this case,
by using the fact that $\varphi$ is $(0,0)$-concave, we get that
$$x\mapsto \frac{x\varphi'(x)}{\varphi(x)}x^{-a}(\varphi(x))^b$$
is also decreasing on $(0,\infty)$ as a product of three positive decreasing functions. In other words, $\varphi$ is $(a,b)$-concave on $(0,\infty)$
for $a\geq 0$ and $b\leq 0$.

By Theorem \ref{thm1}, the function $\varphi$ is $(a,a)$-concave for $a\geq 1$. It follows from the monotonicity of the power mean that $\varphi$ is
$(a,b)$-concave for $b\leq a$ and $a\geq 1$.
\end{proof}

\begin{proof}[\bf Proof of Theorem \ref{nthm3}] Consider the function $f:(0,1)\to(0,\infty)$, defined by $$f(t)=(1-t^p)^{1/q-1}t^{p-2}.$$ Then we have
$$\left(\frac{t^{1-a}f'(t)}{f(t)}\right)'=\frac{p}{q}(q-1)\frac{(p-a)t^{p-a-1}+at^{2p-a-1}}{(1-t^p)^2}+a(p-2)t^{-a-1}.$$
If $a=0$, then we have that $f$ is geometrically convex on $(0,1)$ for $p,q>1$. This in turn implies that the function
$$x\mapsto \lambda(x)=\frac{q}{p}{\rm arcsin}_{p,q}(1)-\frac{q}{p}{\rm arccos}_{p,q}(x)=\int_0^xf(t)dt$$
is also geometrically convex on $(0,1)$.
Moreover, if $a<0$, then taking into account that
$$\left(\frac{t^{1-a}f'(t)}{f(t)}\right)'=\frac{p}{q}(q-1)\frac{t^{p-a-1}(p+a(t^p-1))}{(1-t^p)^2}+a(p-2)t^{-a-1},$$
we obtain that for $p\in(1,2],$ $q>1$ the function $f$ is $(a,0)$-convex. Since $t^{1-a}(1-t^p)^{1/q-1}t^{p-2}\to 0$ as $t\to 0$, according to
Lemma \ref{bari theorem 7} the function $\lambda$ is $(a,b)$-convex for $a<0,$ $b\geq 0$ and $p\in(1,2],$ $q>1$. On the other hand, the function
$$t\mapsto t^{1-a}(1-t^p)^{1/q-1}t^{p-2}=t^{p-a-1}(1-t^p)^{1/q-1}$$
is increasing on $(0,1)$ for $p,q>1$ and $a<0$. Appealing again to Lemma \ref{bari theorem 7} the function $\lambda$ is $(a,b)$-convex
for $a<0,$ $b<0$ and $p\in(1,2],$ $q>1$.

Finally, if we consider the functions $u,v,w:(0,1)\to (0,\infty)$, defined by
$$u(x)=\frac{x\lambda'(x)}{\lambda(x)},\quad v(x)=x^{-a},\quad w(x)=(\lambda(x))^b,$$
then we get the $x\mapsto x^{1-a}\lambda'(x)(\lambda(x))^{b-1}$ is increasing on $(0,1)$ as a product of three positive increasing functions
$u,\,v$ and $w$ when $a\leq0$ and $b\geq 0$. Here we used that $\lambda$ is geometrically convex on $(0,1)$.
\end{proof}


\begin{thebibliography}{width}

\bibitem{AS}
\textsc{M. Abramowitz and I. Stegun, eds.}:
\textit{Handbook of mathematical functions with formulas,
 graphs and mathematical tables.}  National Bureau of Standards, 1964 (Russian translation, Nauka 1979).

\bibitem{avvb}
\textsc{G.D. Anderson, M.K. Vamanamurthy and M. Vuorinen}:
\textit{Conformal invariants, inequalities and quasiconformal
maps.} J. Wiley, 1997.


\bibitem{bari}
\textsc{\'A. Baricz:}
{\it Geometrically concave univariate distributions}, J. Math. Anal. Appl. 363(1) (2010) 182--196.

\bibitem{barbarkvuo}
\textsc{\'A. Baricz, B.A. Bhayo and M. Vuorinen:}
{\it Tur\'an type inequalities for generalized inverse trigonometric functions}, \texttt{arXiv:1305.0938}.

\bibitem{bv}
\textsc{B.A. Bhayo and M. Vuorinen:}
{\it Inequalities for eigenfunctions of the $p$-Laplacian}, \texttt{arXiv.1101.3911}.

\bibitem{bv2}
\textsc{B.A. Bhayo and M. Vuorinen:}
{\it On generalized trigonometric functions  with two parameters},
J. Approx. Theory 164 (2012) 1415--1426.

\bibitem{bv3}
\textsc{B.A. Bhayo and M. Vuorinen:}
{\it Power mean inequality of generalized trigonometric functions}, \texttt{arXiv.1209.0873}.

\bibitem{bem}\textsc{R.J. Biezuner, G. Ercole and E.M. Martins:}
\textit{Computing the first eigenvalue of the $p$-Laplacian
via the inverse power method}, J. Funct. Anal. 257 (2009) 243--270.

\bibitem{be}
\textsc{P.J. Bushell and D.E. Edmunds}:
\textit{Remarks on generalised trigonometric functions}, Rocky Mountain J. Math. 42 (2012) 25--57.

\bibitem{cjw}\textsc{Y.-M. Chu, Y.-P. Jiang and M.-K. Wang}:
\textit{Inequalities for generalized trigonometric and hyperbolic sine functions}, \texttt{arXiv.1212.4681}.

\bibitem{dm}
\textsc{P. Dr\'abek and R. Man\'asevich}:
\textit{On the closed solution to some $p-$Laplacian nonhomogeneous eigenvalue problems}.
Differential Integral Equations 12 (1999) 773--788.

\bibitem{egl}
\textsc{D.E. Edmunds, P. Gurka and J. Lang:}
{\it Properties of generalized trigonometric functions},
J. Approx. Theory 164 (2012) 47--56.


\bibitem{jq}
\textsc{W.-D. Jiang and F. Qi:}
{\it Geometric convexity of the generalized sine and the generalized hyperbolic sine},
\texttt{arXiv.1301.3264}.


\bibitem{kmsv} \textsc{R. Kl\'en, V. Manojlovi\'c, S. Simi\'c and M. Vuorinen:}
{\it Bernoulli inequality and hypergeometric functions}, Proc. Amer. Math. Soc. (in press).

\bibitem{ku} \textsc{M. Kuczma:}
{\it An introduction to the theory of functional equations
 and inequalities}. Cauchy's equation and Jensen's inequality,
  With a Polish summary. Prace Naukowe Uniwersytetu \'Slaskiego
  w Katowicach [Scientific Publications of the University of Silesia],
   489. Uniwersytet \'Slaski, Katowice; Pa\'nstwowe Wydawnictwo
   Naukowe (PWN), Warsaw,  1985.

\bibitem{le} \textsc{J. Lang and D.E. Edmunds:}
{\it Eigenvalues, Embeddings and Generalised Trigonometric Functions},
Lecture Notes in Mathematics, vol. 2016, Springer-Verlag, 2011.

\bibitem{lp} \textsc{P. Lindqvist:}
{\it Some remarkable sine and cosine functions}, Ricerche di Matematica 44 (1995) 269--290.

\bibitem{m} \textsc{M. Mr\v sevi\'c:}
{\it Convexity of the inverse function}, The teaching of mathematics  11 (2008) 21--24.

\bibitem{t}\textsc{S. Takeuchi:}
{\it Generalized Jacobian elliptic functions and their application to bifurcation problems associated with p-Laplacian}, J. Math. Anal. Appl. 385 (2012) 24--35.

\end{thebibliography}
\end{document}